\documentclass[12pt,reqno]{article}

\usepackage[usenames]{color}
\usepackage{amssymb}
\usepackage{amsmath}
\usepackage{amsthm}
\usepackage{amsfonts}
\usepackage{amscd}
\usepackage{graphicx}
\usepackage{mathtools}

\mathtoolsset{showonlyrefs}

\usepackage[colorlinks=true,
linkcolor=webgreen,
filecolor=webbrown,
citecolor=webgreen]{hyperref}

\definecolor{webgreen}{rgb}{0,.5,0}
\definecolor{webbrown}{rgb}{.6,0,0}

\usepackage{color}
\usepackage{fullpage}
\usepackage{float}

\usepackage{graphics}
\usepackage{latexsym}

\begin{document}

\theoremstyle{plain}
\newtheorem{theorem}{Theorem}
\newtheorem{corollary}[theorem]{Corollary}
\newtheorem{proposition}{Proposition}
\newtheorem{lemma}{Lemma}
\newtheorem{example}{Example}
\newtheorem{remark}{Remark}


\newcommand{\braces}{\genfrac{\lbrace}{\rbrace}{0pt}{}}

\begin{center}

\vskip 1cm{\large\bf  
On Touchard's Identity: Generalizations and Related Results}

\vskip 1cm
\large
Kunle Adegoke \\ 
Department of Physics and Engineering Physics\\
Obafemi Awolowo University\\
220005 Ile-Ife \\
Nigeria \\
\href{mailto:adegoke00@gmail.com }{\tt adegoke00@gmail.com}\\

\end{center}

\vskip .2 in

\begin{abstract}
Starting with a known polynomial identity, we derive two generalizations of Touchard's identity concerning Catalan numbers; one obtained using the Beta function and the other via a connection with Stirling numbers of the second kind. We subsequently establish several new combinatorial identities. 
\end{abstract}



\section{Introduction}
Let $C_k=\binom{2k}k/(k+1)$ be the $k$th Catalan number. The celebrated Touchard's identity~\cite{gould77,ratie23,riordan73,shapiro76}:
\begin{equation}\label{touchard}
\sum_{k = 0}^{\left\lfloor {n/2} \right\rfloor } {\binom{n}{{2k}}2^{n - 2k} C_k }  = C_{n + 1}, 
\end{equation}
is subsumed at $x=0$ by the following identity~\cite[Equation (54)]{adegoke26}:
\begin{equation}\label{fe8atkt}
\sum_{k = 0}^{\left\lfloor {n/2} \right\rfloor } {\binom{{n}}{{2k}}2^{ - 2k} C_k \left( {1 - x} \right)^{2k} }=\sum_{k = 0}^n {\binom{{n}}{k}2^{ - k} C_{k + 1} \left( {1 - x} \right)^k x^{n - k} } .
\end{equation}
Based on~\eqref{fe8atkt}, we will derive the following generalization of~\eqref{touchard}:

\begin{equation}\label{gen_touchard}
\sum_{k = 0}^{\left\lfloor {n/2} \right\rfloor } {\binom{{n}}{{2k}}2^{n - 2k} C_k \binom{{2k + r}}{s}^{ - 1} }  = \frac{s}{{n + r}}\sum_{k = 0}^n {\binom{{n}}{k}2^{n - k} C_{k + 1} \binom{{n + r - 1}}{{k + r - s}}^{ - 1} } ,
\end{equation}
which holds for complex numbers $r$ and $s$ that are not negative integers.

Note that~\eqref{gen_touchard} reduces to~\eqref{touchard} when $s=0$, since
\begin{equation}\label{pnfixx7}
\lim_{s\to 0} s\binom{p - 1}{p - s}^{-1} = \lim_{s\to 0} s\,\frac ps\binom p{p-s}^{-1}=p.
\end{equation}
By linking~\eqref{fe8atkt} with Stirling numbers of the second kind, defined for non-negative integers $m$ and $n$ through the generating function
\begin{equation}
\sum_{k = 0}^n {\binom{{n}}{k}\braces{{ m}}{k}k!}  = n^m ,
\end{equation}
and having the explicit representation
\begin{equation}\label{bluqqd7}
\braces{{ m}}{n} = \frac{{( - 1)^n }}{{n!}}\sum_{k = 0}^n {( - 1)^k \binom{{n}}{k}k^m },
\end{equation}
we will derive another generalization of Touchard's identity, namely,
\begin{equation}
\sum_{k = 0}^{\left\lfloor {n/2} \right\rfloor } {\binom{{n}}{{2k}}2^{n - 2k} \left( {n - 2k} \right)^r C_k }  = \sum_{k = 0}^r {2^k \binom{{n}}{k}\braces{{ r}}{k}k!C_{n - k + 1} } ,
\end{equation}
which holds for every non-negative integer $r$.

Identity~\eqref{fe8atkt} has further interesting immediate consequences. Evaluation at $x=2$ produces the following self-inverse sequence binomial-transform identity:
\begin{equation}
\sum_{k = 0}^n {( - 1)^k \binom{{n}}{k}2^{ - 2k} C_{k + 1} }  = 2^{ - 2n} C_{n + 1} ,
\end{equation}
which was reported by Donaghey~\cite{donaghey76}. Setting $x=-1$, and $x=3$, in turn, in~\eqref{fe8atkt}, produces the following interesting identity:
\begin{equation}
\sum_{k = 0}^n {( - 1)^k \binom{{n}}{k}C_{k + 1} }  = ( - 3)^n \sum_{k = 0}^n {( - 1)^k \binom{{n}}{k}3^{ - k} C_{k + 1} }  = ( - 1)^n \sum_{k = 0}^n {\binom{{n}}{{2k}}C_k } .
\end{equation}

Binomial coefficients are defined, for non-negative integers $i$ and $j$, by
\begin{equation}
\binom ij=
\begin{cases}
\dfrac{{i!}}{{j!(i - j)!}}, & \text{$i \ge j$};\\
0, & \text{$i<j$}.
\end{cases}
\end{equation}
The definition is extended to complex numbers $r$ and $s$ by
\begin{equation}\label{y89d722}
\binom rs= \frac{{\Gamma (r + 1)}}{{\Gamma (s + 1)\Gamma (r - s + 1)}},
\end{equation}
where the Gamma function, $\Gamma(z)$, is defined for $\Re(z)>0$ by
\begin{equation}
\Gamma (z) = \int_0^\infty  {e^{ - t} t^{z - 1}dt}  = \int_0^\infty  {\left( {\log (1/t)} \right)^{z - 1}dt},
\end{equation}
and is extended to the rest of the complex plane, excluding the non-positive integers, by analytic continuation.

Relationships between binomial coefficients can be established by using~\eqref{y89d722}, together with the following identities:
\begin{equation}\label{v7g54vl}
\Gamma \left( {z + \frac{1}{2}} \right) = \sqrt \pi\, 2^{-2z}\binom{{2z}}{z}\,\Gamma \left( {z + 1} \right),
\end{equation}
and
\begin{equation}\label{ho6212z}
\Gamma \left( { - z + \frac{1}{2}} \right) = ( - 1)^z\, 2^{2z} \,\binom{{2z}}{z}^{ - 1} \frac{{\sqrt \pi  }}{{\Gamma \left( {z + 1} \right)}}.
\end{equation}

Harmonic numbers, $H_n$, and odd harmonic numbers, $O_n$, are defined for non-negative integers $n$ by
\begin{equation}
H_n=\sum_{k=1}^n\frac1k,\quad O_n=\sum_{k=1}^n\frac1{2k-1},\quad H_0=0=O_0.
\end{equation}
We will make frequent use of the following identity:
\begin{equation}\label{nerd5x1}
H_{n-1/2}=2O_n-2\ln 2,
\end{equation}
which extends harmonic numbers to half-integer arguments and which is a consequence of the link between harmonic numbers, odd harmonic numbers and the digamma function.

\section{A generalization of Touchard's identity and related results}\label{touch}
Our first generalization of Touchard's identity~\eqref{touchard} is stated in Theorem~\ref{thm.touchard}. We require the integration formulas stated in Lemma~\ref{integrals}, which are variations on the well-known Beta function.
\begin{lemma}\label{integrals}
If $r$ and $s$ are complex numbers, then
\begin{equation}
\int_0^1 {y^{r + k - s} \left( {1 - y} \right)^{s - 1} dy}  = \frac{1}{s}\binom{{k + r}}{s}^{ - 1} 
\end{equation}
and
\begin{equation}
\int_0^1 {y^{s + k} \left( {1 - y} \right)^{r - k} dy}  = \frac{1}{{r + s + 1}}\binom{{r + s}}{{k + s}}^{ - 1} .
\end{equation}
\end{lemma}

\begin{theorem}\label{thm.touchard}
If $r$ and $s$ are complex numbers that are not negative integers, then
\begin{equation}
\sum_{k = 0}^{\left\lfloor {n/2} \right\rfloor } {\binom{{n}}{{2k}}2^{- 2k} C_k \binom{{2k + r}}{s}^{ - 1} }  = \frac{s}{{n + r}}\sum_{k = 0}^n {\binom{{n}}{k}2^{ - k} C_{k + 1} \binom{{n + r - 1}}{{k + r - s}}^{ - 1} }.
\end{equation}
\end{theorem}
\begin{proof}
Using~\eqref{fe8atkt}, write
\begin{align}
\sum_{k = 0}^{\left\lfloor {n/2} \right\rfloor } {\binom{{n}}{{2k}}2^{ - 2k} C_k x^{r - s + 2k} \left( {1 - x} \right)^{s - 1} }  = \sum_{k = 0}^n {\binom{{n}}{k}2^{ - k} C_{k + 1} x^{k + r - s} \left( {1 - x} \right)^{n - k + s - 1} } ,
\end{align}
and integrate term-wise from $0$ to $1$, using Lemma~\ref{integrals}.
\end{proof}
The rest of this section is devoted to deriving some consequences of~\eqref{gen_touchard}.
\begin{proposition}
If $n$ is a non-negative integer, then
\begin{align}\label{v2skdnk}
\sum_{k = 0}^{\left\lfloor {n/2} \right\rfloor } {\binom{{n}}{{2k}}2^{ - 6k} C_k C_{2k} }  = \frac{1}{{\left( {n + 1} \right)2^{2n} }}\sum_{k = 0}^n { \left( {2\left( {n - k} \right) + 1} \right)\binom{{2\left( {n - k} \right)}}{{n - k}}\binom{{2k}}{k}2^{ - k}C_{k + 1} }
\end{align}
and
\begin{equation}\label{ymbrlhu}
\sum_{k = 0}^{\left\lfloor {n/2} \right\rfloor } {\binom{{n}}{{2k}}2^{2k} \frac{{C_k }}{{C_{2k} }}}  = 3\binom{{2n}}{n}^{ - 1} \sum_{k = 0}^n {\frac{{2^k \left( {k + 1} \right)}}{{\left( {2\left(n - k\right) - 1} \right)\left( {2\left(n - k\right) - 3} \right)}}\binom{{2\left( {n - k} \right)}}{{n - k}}C_{k + 1} } .
\end{equation}
\end{proposition}
\begin{proof}
To derive~\eqref{v2skdnk}, set $r=1$ and $s=3/2$ in~\eqref{gen_touchard} and use
\begin{equation}\label{n86netd}
\binom{{2k + 1}}{{3/2}} = \frac{4}{{3\sqrt \pi  }}\frac{{\Gamma \left( {2k + 2} \right)}}{{\Gamma \left( {2k + 1/2} \right)}} =\frac{4}{{3\pi }}\frac{{2^{4k} }}{{C_{2k} }}
\end{equation}
and
\begin{equation}\label{z926t63}
\binom{{n}}{{k - 1/2}}  = \frac{{\Gamma \left( {n + 1} \right)}}{{\Gamma \left( {k + \frac{1}{2}} \right)\Gamma \left( {n - k + \frac{3}{2}} \right)}}
= \frac{{2^{2n + 1} }}{{\pi \left( {2\left( {n - k} \right) + 1} \right)}}\binom{{n}}{k}\binom{{2\left( {n - k} \right)}}{{n - k}}^{ - 1} \binom{{2k}}{k}^{ - 1} .
\end{equation}
Identity~\eqref{ymbrlhu} is obtained by setting $r=-1/2$ and $s=-3/2$ in~\eqref{gen_touchard} and using
\begin{equation}
\binom{{2k - 1/2}}{{ - 3/2}} =  - 2^{ - 4k - 1} C_{2k} 
\end{equation}
and
\begin{equation}
\binom{{n - 3/2}}{{k + 1}} = \frac{{2^{ - 2k} }}{2}\frac{{\left( {2\left( {n - k} \right) - 1} \right)\left( {2\left( {n - k} \right) - 3} \right)}}{{\left( {2n - 1} \right)\left( {k + 1} \right)}}\binom{{2n}}{n}\binom{{n}}{k}\binom{{2\left( {n - k} \right)}}{{n - k}}.
\end{equation}
\end{proof}

\begin{proposition}
If $n$ is a non-negative integer and $r$ is a complex number that is not a non-positive integer, then
\begin{equation}\label{t7drac2}
\sum_{k = 0}^{\left\lfloor {n/2} \right\rfloor } {\binom{{n}}{{2k}}\frac{{2^{ - 2k}  }}{{2k + r}}\,C_k}  = \frac{{1}}{{n + r}}\binom{{n + r - 1}}{{n}}^{ - 1}\sum_{k = 0}^n {2^{ - k} C_{k + 1} \binom{{k + r - 1}}{k}} .
\end{equation}
In particular,
\begin{align}
\sum_{k = 0}^{\left\lfloor {n/2} \right\rfloor } {\binom{{n}}{{2k}}\frac{{2^{ - 2k} }}{{2k + 1}}}\, C_k  &= \frac{1}{{n + 1}}\sum_{k = 0}^n {2^{ - k} C_{k + 1} }\label{u7r6uj7},\\
\sum_{k = 0}^{\left\lfloor {n/2} \right\rfloor } {\binom{{n}}{{2k}}\frac{{2^{ - 2k} }}{{4k + 1}}}\, C_k  &= \frac{{2^{2n} }}{{2n + 1}}\binom{{2n}}{n}^{ - 1} \sum_{k = 0}^n {2^{ - 3k} \binom{{2k}}{k}C_{k + 1} }\label{c3ntdo9}, \\
\sum_{k = 0}^{\left\lfloor {n/2} \right\rfloor } {\binom{{n}}{{2k}}\frac{{2^{ - 2k} }}{{4k - 1}}\,C_k }  &= 2^{2n} \binom{{2n}}{n}^{ - 1} \left( {2\sum_{k = 1}^n {2^{ - 3k} C_{k - 1} C_{k + 1} }  - 1} \right)\label{bld3mys},
\end{align}
and
\begin{equation}\label{ou4dvmp}
\sum_{k = 0}^{\left\lfloor {n/2} \right\rfloor } {\binom{{n}}{{2k}}\frac{{2^{ - 2k} }}{{2k - 1}}\,C_k }  = n\left( { - 1 + \sum_{k = 2}^n {2^{ - k}\, \frac{{C_{k + 1} }}{{k\left( {k - 1} \right)}}} } \right).
\end{equation}
\end{proposition}
\begin{proof}
Set $s=1$ in~\eqref{gen_touchard} and use
\begin{equation}\label{j5c96vt}
\binom{{n + r - 1}}{{k + r - 1}}\binom{{k + r - 1}}{k} = \binom{{n}}{k}\binom{{n + r - 1}}{n}.
\end{equation}
Identity~\eqref{u7r6uj7} corresponds to setting $r=1$ in~\eqref{t7drac2} while Identity~\eqref{c3ntdo9} is obtained by evaluating~\eqref{t7drac2} at $r=1/2$ and making use of
\begin{equation}\label{p1jsq1n}
\binom{{k - 1/2}}{k} = 2^{ - 2k} \binom{{2k}}{k}.
\end{equation}
Identity~\eqref{bld3mys} is derived by shifting the summation index on the right-hand side of~\eqref{t7drac2}, setting $r=-1/2$ and using
\begin{equation}
\binom{{k - 1/2}}{{k + 1}} =  - \frac{{C_k }}{{2^{2k + 1} }}.
\end{equation}
To prove~\eqref{ou4dvmp}, write~\eqref{t7drac2} as
\begin{align}
&\sum_{k = 0}^{\left\lfloor {n/2} \right\rfloor } {\binom{{n}}{{2k}}\frac{{2^{ - 2k} }}{{2k + r}}\,C_k }\\
&\qquad  = \frac{1}{{n + r}}\left( {\left( {1 + r} \right)\binom{{n + r - 1}}{n}^{ - 1}  + \sum_{k = 2}^n {2^{ - k} C_{k + 1} \binom{{k + r - 1}}{k}\binom{{n + r - 1}}{n}^{ - 1} } } \right),
\end{align}
and use
\begin{equation}
\lim_{r\to -1}\left( {1 + r} \right)\binom{{n + r - 1}}{n}^{ - 1}  =  - n\left( {n - 1} \right)
\end{equation}
and
\begin{equation}
\lim_{r\to -1}\binom{{k + r - 1}}{k}\binom{{n + r - 1}}{n}^{ - 1}  = \frac{{n\left( {n - 1} \right)}}{{k\left( {k - 1} \right)}},
\end{equation}
since
\begin{equation}
\left( {1 + r} \right)\binom{{n + r - 1}}{n}^{ - 1}  = \frac{{\left( {n + r} \right)\left( {n + r + 1} \right)}}{r}\binom{{n + r + 1}}{n}^{ - 1} 
\end{equation}
and
\begin{equation}
\binom{{k + r - 1}}{k}\binom{{n + r - 1}}{n}^{ - 1}  = \frac{{n + r}}{{k + r}}\binom{{n}}{{n + r}}\binom{{k}}{{k + r}}^{ - 1} .
\end{equation}
\end{proof}

\begin{proposition}
If $n$ is a non-negative integer and $s$ is a complex number that is not a negative integer, then
\begin{equation}\label{k7vgfov}
\sum_{k = 0}^{\left\lfloor {n/2} \right\rfloor } {\binom{{n}}{{2k}}} 2^{ - 2k} C_k \binom{{2k + s}}{s}^{ - 1}  = \frac{s}{{n + s}}\sum_{k = 0}^n {\binom{{n}}{k}2^{ - k} C_{k + 1} \binom{{n - 1 + s}}{k}^{ - 1} } .
\end{equation}
In particular,
\begin{align}
&\sum_{k = 0}^{\left\lfloor {n/2} \right\rfloor } {\binom{{n}}{{2k}}} 2^{2k} C_k \binom{{4k}}{{2k}}^{ - 1}\\  &\qquad= \frac{2^n3n}{{\left( {n + 1} \right)\left( {n + 2} \right)}} - \sum_{k = 0}^{n - 2} {\frac{{2^k }}{{2n - 2k - 1}}\binom{{n}}{k}^2 C_{k + 1}\binom{2k}k^{-1}\binom{2n}{2k}^{-1} }\label{e5hxpsc} .
\end{align}
\end{proposition}
\begin{proof}
Set $r=s$ in~\eqref{gen_touchard}. To derive~\eqref{e5hxpsc}, set $s=-1/2$ in~\eqref{k7vgfov}, write
\begin{align}
\sum_{k = 0}^n {\binom{{n}}{k}2^{ - k} C_{k + 1} \binom{{n - 2 + 1/2}}{k}^{ - 1} }  &= \sum_{k = 0}^{n - 2} {\binom{{n}}{k}2^{ - k} C_{k + 1} \binom{{n - 2 + 1/2}}{k}^{ - 1} }\\
&\qquad  + \binom{{n}}{{n - 1}}2^{ - (n - 1)} C_n \binom{{n - 2 + 1/2}}{{n - 1}}^{-1}\\
&\qquad\qquad + 2^{ - n} C_{n + 1} \binom{{n - 2 + 1/2}}{n}^{-1}
\end{align}
and use
\begin{equation}
\binom{{n - 2 + 1/2}}{k} = 2^{ - 2k} \binom{{2n - 3}}{{2k}}\binom{{2k}}{k}\binom{{n - 2}}{k}^{ - 1} 
\end{equation}
and
\begin{equation}
\binom{{2k - 1/2}}{{ - 1/2}} = \binom{{2k - 1/2}}{{2k}} = 2^{ - 4k} \binom{{4k}}{{2k}}.
\end{equation}
\end{proof}

\begin{proposition}
If $n$ and $r$ are non-negative integers and $s$ is a complex number such that $s-n$ is not a negative integer, then
\begin{equation}\label{z5mk6gd}
\sum_{k = 0}^{\left\lfloor {n/2} \right\rfloor } {\binom{{n}}{{2k}}2^{ - 2k} C_k \binom{{n - s - 1}}{{2k + r}}^{ - 1} }  = ( - 1)^{r - 1}\, \frac{{n - s}}{{r + s}}\,\sum_{k = 0}^n {\binom{{n}}{k}2^{ - k} C_{k + 1} \binom{{r + s - 1}}{{k + r}}^{-1}} .
\end{equation}
In particular,
\begin{equation}\label{x5itttn}
\sum_{k = 0}^{\left\lfloor {n/2} \right\rfloor } {\frac{{2^{2k} }}{{\left( {2n - 4k - 1} \right)\left( {2n - 4k - 3} \right)}}\binom{{n}}{{2k}}^2 \binom{{2k}}{k}\binom{{2n}}{{4k}}^{ - 1} \frac{1}{{C_{2k} }}}  = ( - 1)^n\, \frac{{2n + 1}}{3}.
\end{equation}
\end{proposition}
\begin{proof}
Write $r-n$ for $r$ and $s-n$ for $s$ in~\eqref{gen_touchard} and write $r+s$ for $r$ in the resulting equation, noting that
\begin{equation}
\binom{2k+r+s-n}{2k+r}=(-1)^r\binom{n-s-1}{2k+r}.
\end{equation}
To derive~\eqref{x5itttn}, set $r=2$ and $s=-1/2$ in~\eqref{z5mk6gd} and use
\begin{equation}
\binom{{1/2}}{{k + 2}} = \frac{{( - 1)^{k + 1} }}{{2^{2k + 3} }}C_{k + 1} 
\end{equation}
and
\begin{equation}
\binom{{n - 1/2}}{{2k + 2}} = \frac{{2^{ - 4k} }}{8}\frac{{\left( {2n - 4k - 1} \right)\left( {2n - 4k - 3} \right)}}{{k + 1}}\binom{{2n}}{{4k}}\binom{{n}}{{2k}}^{ - 1} C_{2k} .
\end{equation}
\end{proof}
\begin{proposition}
If $n$ is a non-negative integer and $r$ and $s$ are complex numbers that are not negative integers, then
\begin{align}
&\sum_{k = 0}^{\left\lfloor {n/2} \right\rfloor } {\binom{{n}}{{2k}}2^{ - 2k} C_k H_{2k + r - s} \binom{{2k + r}}{s}^{ - 1} }\\
&\qquad  = \frac{s}{{n + r}}\sum_{k = 0}^n {\binom{{n}}{k}2^{ - k} C_{k + 1} \left( {H_{k + r - s}  + H_{s - 1}  - H_{n + s - k - 1} } \right)\binom{{n + r - 1}}{{k + r - s}}^{-1}}\label{j5hkdla} .
\end{align}
In particular,
\begin{equation}\label{kzddtlk}
\sum_{k = 0}^{\left\lfloor {n/2} \right\rfloor } {\binom{{n}}{{2k}}2^{ - 2k} C_k H_{2k + r} }  = 2^{ - n} C_{n + 1} H_{n + r}  - \sum_{k = 0}^{n - 1} {\binom{{n}}{k}\frac{2^{ - k}}{n-k} C_{k + 1} \binom{{n + r}}{{k + r}}^{ - 1} }
\end{equation}
and
\begin{align}
&\sum_{k = 0}^{\left\lfloor {n/2} \right\rfloor } {\binom{{n}}{{2k}}\frac{{2^{ - 2k} }}{{2k + r}}\,C_k H_{2k + r - 1} }\\
&\qquad  = \frac{1}{{n + r}}\binom{{n + r - 1}}{n}^{ - 1} \sum_{k = 0}^n {2^{ - k} C_{k + 1} \left( {H_{k + r - 1}  - H_{n - k} } \right)\binom{{k + r - 1}}{k}} \label{ck18ui4}.
\end{align}
\end{proposition}

\begin{proof}
Differentiate~\eqref{gen_touchard} with respect to $s$ to obtain
\begin{align}
&\sum_{k = 0}^{\left\lfloor {n/2} \right\rfloor } {\binom{{n}}{{2k}}2^{ - 2k} C_k \left( {H_{2k + r - s}  - H_s } \right)\binom{{2k + r}}{s}^{ - 1} }\\
&\qquad  =  - \frac{1}{{n + r}}\sum_{k = 0}^n {\binom{{n}}{k}2^{ - k} C_{k + 1} \binom{{n + r - 1}}{{k + r - s}}^{ - 1} }\\
&\qquad\qquad  + \frac{s}{{n + r}}\sum_{k = 0}^n {\binom{{n}}{k}2^{ - k} C_{k + 1} \left( {H_{k + r - s}  - H_{n + s - k - 1} } \right)\binom{{n + r - 1}}{{k + r - s}}^{-1}}, 
\end{align}
from which~\eqref{j5hkdla} follows after using~\eqref{gen_touchard} again. To derive~\eqref{kzddtlk}, note that $sH_{s-1}=sH_s-1$, write the right-hand side of~\eqref{j5hkdla} as
\begin{align}
&\frac{1}{{n + r}}\sum_{k = 0}^{n - 1} {\binom{{n}}{k}2^{ - k} C_{k + 1} \binom{{n + r - 1}}{{k + r - s}}^{ - 1} \left( {sH_{k + r - s}  + sH_s  - 1 - sH_{n + s - k - 1} } \right)}\\
&\qquad  + \frac{1}{{n + r}}2^{ - n} sC_{n + 1} H_{n + r - s} \binom{{n + r - 1}}{{n + r - s}}^{ - 1} 
\end{align}
and take the limit as $s$ approaches zero, using~\eqref{pnfixx7}. Identity~\eqref{ck18ui4} is obtained by setting $s=1$ in~\eqref{j5hkdla} and making use of~\eqref{j5c96vt}.
\end{proof}
\begin{proposition}
If $n$ is a non-negative integer, then
\begin{equation}\label{uj0xxbk}
\sum_{k = 0}^{\left\lfloor {n/2} \right\rfloor } {\binom{{n}}{{2k}}2^{ - 2k} C_k H_{2k} }  = 2^{ - n} C_{n + 1} H_n  - \sum_{k = 0}^{n - 1} {\frac{{2^{ - k} C_{k + 1} }}{{n - k}}} 
\end{equation}
and
\begin{align}
&\sum_{k = 0}^{\left\lfloor {n/2} \right\rfloor } {\binom{{n}}{{2k}}2^{ - 2k} C_k O_{2k} }\\
&\qquad= 2^{ - n} C_{n + 1} O_n - 2^{ - n - 1} \sum_{k = 1}^n {\frac{{2^{3k} }}{k}\binom{{n}}{k}^2 \binom{{2k}}{k}^{ - 1} \binom{{2n}}{{2k}}^{ - 1} C_{n - k + 1} }\label{uyctab9} .
\end{align}
\end{proposition}
\begin{proof}
Set $r=0$ in~\eqref{kzddtlk} to obtain~\eqref{uj0xxbk}, and $r=-1/2$ to derive~\eqref{uyctab9} after some algebra.
\end{proof}
\begin{proposition}
If $n$ is a non-negative integer, then
\begin{align}
&\sum_{k = 0}^{\left\lfloor {n/2} \right\rfloor } {\binom{{n}}{{2k}}2^{ - 6k} C_k C_{2k} O_{2k} }\\
&\qquad  = \frac{1}{{2^{2n} \left( {n + 1} \right)}}\sum_{k = 0}^n {\left( {2\left( {n - k} \right) + 1} \right) \binom{{2\left( {n - k} \right)}}{{n - k}}\binom{{2k}}{k}2^{ - k}C_{k + 1} \left( {O_k  - O_{n - k + 1}  + 1} \right)} .
\end{align}
\end{proposition}
\begin{proof}
Set $r=1$ and $s=3/2$ in~\eqref{j5hkdla} and equate rational parts after using~\eqref{n86netd} and~\eqref{z926t63}.
\end{proof}
\begin{proposition}
If $n$ is a non-negative integer, then
\begin{equation}\label{zvgyvn4}
\sum_{k = 0}^{\left\lfloor {n/2} \right\rfloor } {\binom{{n}}{{2k}}\frac{{2^{ - 2k} }}{{2k + 1}}\,C_k H_{2k} }  = \frac{1}{{n + 1}}\sum_{k = 0}^n {2^{ - k} C_{k + 1} \left( {H_k  - H_{n - k} } \right)} 
\end{equation}
and
\begin{equation}\label{kxh1tad}
\sum_{k = 0}^{\left\lfloor {n/2} \right\rfloor } {\binom{{n}}{{2k}}\frac{{2^{ - 2k} }}{{4k + 1}}\,C_k O_{2k} }  = \frac{{2^{2n - 1} }}{{2n + 1}}\binom{{2n}}{n}^{ - 1} \sum_{k = 0}^n {2^{ - 3k} C_{k + 1} \binom{{2k}}{k}\left( {2O_k  - H_{n - k} } \right)}.
\end{equation}
\end{proposition}
\begin{proof}
Identity~\eqref{zvgyvn4} is the evaluation of~\eqref{ck18ui4} at $r=1$. Identity~\eqref{kxh1tad} is obtained by plugging $r=1/2$ in~\eqref{ck18ui4}, making use of~\eqref{nerd5x1} and~\eqref{p1jsq1n} and equating rational parts.
\end{proof}

\section{A generalization involving Stirling numbers of the second kind}

\begin{lemma}\label{m-derivatives}
If $m$ and $u$ are non-negative integers, then
\begin{equation}\label{eflzmiz}
\left. {\frac{{d^m }}{{dx^m }}\left( {{1 - e^x } } \right)^u} \right|_{x = 0}  = \sum_{p = 0}^u {( - 1)^p \binom{{u}}{p}p^m }=(-1)^uu!\braces mu.
\end{equation}

\end{lemma}

\begin{proof}

Since
\begin{equation}
\left( {1 - e^x } \right)^u  = \sum_{p = 0}^u {( - 1)^p \binom{{u}}{p}e^{px} } ,
\end{equation}
we have
\begin{equation}
\frac{{d^m }}{{dx^m }} {\left( {1 - e^x } \right)^u } = \sum_{p = 0}^u {( - 1)^p \binom{{u}}{p}p^m e^{px} };
\end{equation}
and hence~\eqref{eflzmiz} on account of~\eqref{bluqqd7}.

\end{proof}

\begin{theorem}
If $n$ and $r$ are non-negative integers, then
\begin{equation}
\sum_{k = 0}^{\left\lfloor {n/2} \right\rfloor } {\binom{{n}}{{2k}}2^{n - 2k} \left( {n - 2k} \right)^r C_k }  = \sum_{k = 0}^r {2^k \binom{{n}}{k}\braces{{ r}}{k}k!C_{n - k + 1} }.
\end{equation}

\end{theorem}
\begin{proof}
By writing $1/x$ for $x$, Identity~\eqref{fe8atkt} can also be written as
\begin{equation}\label{cmtd13h}
\sum_{k = 0}^{\left\lfloor {n/2} \right\rfloor } {\binom{{n}}{{2k}}2^{ - 2k} C_k x^{n - 2k} }=\sum_{k = 0}^n {( - 1)^{n - k} \binom{{n}}{k}2^{ - k} C_{k + 1} \left( {1 - x} \right)^{n - k} } .
\end{equation}
Now replace $x$ with $\exp x$, differentiate the resulting expression $r$ times with respect to $x$ and evaluate at $x=0$, making use of Lemma~\ref{m-derivatives}.
\end{proof}

\section{Related results via binomial transform}
Two sequences of complex numbers $(s_k)$ and $(\sigma_k)$, $k=0,1,2,\ldots$, are called a binomial- transform pair if
\begin{equation}\label{bt}
s_n  = \sum_{k = 0}^n {( - 1)^k \binom{{n}}{k}\sigma _k },
\end{equation}
for every non-negative integer $n$. It is known that
\begin{equation}
s_n  = \sum_{k = 0}^n {( - 1)^k \binom{{n}}{k}\sigma _k }\iff \sigma_n  = \sum_{k = 0}^n {( - 1)^k \binom{{n}}{k}s_k } .
\end{equation}

\begin{lemma}[{\cite[Lemma 2.1]{adegoke25}}]\label{qqdca04}
Let $\left\{(a_k),(\alpha_k)\right\}$, $k=0,1,2,\ldots$, be a binomial-transform pair. Let $\mathcal L_x$ be a linear operator defined by $\mathcal L_x(x^j)=a_j$ for every complex number $x$ and every non-negative integer~$j$. Then $\mathcal L_x((1-x)^j)=\alpha_j$.
\end{lemma}

\begin{theorem}
Let $\left\{(a_k),(\alpha_k)\right\}$, $k=0,1,2,\ldots$, be a binomial-transform pair. Then
\begin{equation}\label{bt1}
\sum_{k = 0}^n {( - 1)^{n - k} \binom{{n}}{k}2^{ - k} C_{k + 1}a_{n-k} }  = \sum_{k = 0}^{\left\lfloor {n/2} \right\rfloor } {\binom{{n}}{{2k}}2^{ - 2k} C_k \alpha_{n - 2k} } .
\end{equation}

\end{theorem}
\begin{proof}
Write~\eqref{cmtd13h} as
\begin{equation}\label{k24wp0d}
\sum_{k = 0}^n {( - 1)^{n - k} \binom{{n}}{k}2^{ - k} C_{k + 1} x^{n - k} }  = \sum_{k = 0}^{\left\lfloor {n/2} \right\rfloor } {\binom{{n}}{{2k}}2^{ - 2k} C_k \left( {1 - x} \right)^{n - 2k} } 
\end{equation}
and apply Lemma~\ref{qqdca04}.
\end{proof}

\begin{remark}
It is sometimes useful to write~\eqref{bt1} in the following form:
\begin{align}
&\sum_{k = 0}^{\left\lfloor {n/2} \right\rfloor } {\binom{{n}}{{2k}}2^{n - 2k} C_k a_{n - 2k} }  \\
&\qquad= \sum_{k = 0}^{\left\lfloor {n/2} \right\rfloor } {\binom{{n}}{{2k}}2^{2k} C_{n - 2k + 1} \alpha _{2k} }  - \sum_{k = 1}^{\left\lceil {n/2} \right\rceil } {\binom{{n}}{{2k - 1}}2^{2k - 1} C_{n - 2k + 2} \alpha _{2k - 1} }\label{bt2} .
\end{align}
\end{remark}

\begin{proposition}
If $n$ is a non-negative integer, then
\begin{equation}\label{pn5otn6}
\sum_{k = 0}^n {( - 1)^k \binom{{n}}{k}2^{ - k} C_{k + 1} } 
=\begin{cases}
 2^{ - n} C_{n/2},&\text{if $n$ is even;}  \\ 
 0,&\text{if $n$ is odd.} \\ 
 \end{cases}
\end{equation}
\end{proposition}
\begin{proof}
Use the binomial-transform pair $\{a_k=1,\alpha_k=\delta_{k0}\}$ in~\eqref{bt1}. Here $\delta_{ij}$ is Kronecker's delta symbol.
\end{proof}

\begin{remark}
Identity~\eqref{pn5otn6} was obtained by Suleiman and Sury~\cite{sulei23}.
\end{remark}

\begin{proposition}
If $n$ is a non-negative integer, then
\begin{equation}\label{pr3ehv4}
\sum_{k = 0}^n {( - 1)^{n - k} \binom{{n}}{k}\frac{{2^{n - 2k} }}{{n - k + 1}}\,C_{k + 1} } 
 = \begin{cases}
 \sum\limits_{k = 0}^{n/2 } {\dbinom{{n}}{{2k}}\dfrac{{2^{ - 2k} }}{{n - 2k + 1}}\,C_k },&\text{if $n$ is even;}  \\ 
 0,&\text{if $n$ is odd;} \\ 
 \end{cases}
\end{equation}
and
\begin{align}
&\sum_{k = 0}^n {( - 1)^{n - k} \binom{{n}}{k}\frac{{2^{n - 2k} }}{{n - k + 1}}\,C_{k + 1}H_{n-k} }\\ 
&\qquad = \begin{cases}
 -2\sum\limits_{k = 0}^{(n-1)/2 } {\dbinom{{n}}{{2k}}\dfrac{{2^{ - 2k} }}{{n - 2k + 1}}\,C_kO_{(n-2k+1)/2} },&\text{if $n$ is odd;}  \\ 
 0,&\text{if $n$ is even.} \label{e4himaz}
 \end{cases}
\end{align}

\end{proposition}
\begin{proof}
The second identity in~\cite[Proposition 13.3]{adegoke25} is
\begin{equation}\label{n2x6xs0}
\sum_{k = 0}^n {( - 1)^k \binom{{n}}{k}\frac{{2^k }}{{k + 1}}} 
=\begin{cases}
 \dfrac1{n + 1},&\text{if $n$ is even;}  \\ 
 0,&\text{if $n$ is odd;}  \\ 
 \end{cases} 
\end{equation}
where the binomial-transform pair $(a_k)$ and $(\alpha_k)$, $k=0,1,2,\ldots$, with 
\begin{equation}\label{f8pnqqh}
a_k=\frac{2^k}{k+1}\text{ and }\alpha_k=\frac{1+(-1)^k}{2\,\left(k+1\right)},
\end{equation}
can be identified. Use these in~\eqref{bt1} to obtain~\eqref{pr3ehv4}. From~\cite[Proposition 13.4]{adegoke25}:
\begin{equation}
\sum_{k = 0}^n {( - 1)^k \binom{{n}}{k}\frac{{2^k H_k }}{{k + 1}}}
=\begin{cases}
 0,&\text{if $n$ is even;} \\ 
  - \dfrac{{2O_{(n + 1)/2} }}{{n + 1}},&\text{if $n$ is odd;} \\ 
 \end{cases} 
\end{equation}
we recognize the binomial-transform pair $(a_k)$ and $(\alpha_k)$, $k=0,1,2,\ldots$, where
\begin{equation}\label{lfr23eo}
a_k=\frac{2^kH_k}{k+1}\text{ and }\alpha_k=-\frac{1-(-1)^k}{\,\left(k+1\right)}\,O_{(k+1)/2}.
\end{equation}
Use of these in~\eqref{bt1} gives~\eqref{e4himaz}.
\end{proof}

\begin{proposition}
If $n$ is a non-negative integer, then
\begin{equation}\label{mydk3ex}
\sum_{k = 0}^{\left\lfloor {n/2} \right\rfloor } {\binom{{n}}{{2k}}\frac{{2^{2\left( {n - 2k} \right)} }}{{n - 2k + 1}}\,C_k }  = \sum_{k = 0}^{\left\lfloor {n/2} \right\rfloor } {\binom{{n}}{{2k}}\frac{{2^{2k} }}{{2k + 1}}\,C_{n - 2k + 1} } 
\end{equation}
and
\begin{equation}\label{eaj49tl}
\sum_{k = 0}^{\left\lfloor {n/2} \right\rfloor } {\binom{{n}}{{2k}}\frac{{2^{2\left( {n - 2k} \right)} }}{{n - 2k + 1}}\,C_k H_{n - 2k} }  = \frac{1}{2}\sum_{k = 1}^{\left\lceil {n/2} \right\rceil } {\binom{{n}}{{2k - 1}}\frac{{2^{2k} }}{k}\,C_{n - 2k + 2} O_k } .
\end{equation}
\end{proposition}
\begin{proof}
Use~\eqref{f8pnqqh} in~\eqref{bt2} to obtain~\eqref{mydk3ex} and~\eqref{lfr23eo} to get~\eqref{eaj49tl}.
\end{proof}

\section{Additional results}
The methods described in this paper can also be used to derive generalizations of another well-known combinatorial identity, Knuth's old sum or the Reed Dawson identity:
\begin{equation}\label{knuth}
\sum_{k = 0}^n {( - 1)^k \binom{{n}}{k}2^{ - k} \binom{{2k}}{k}}
 =  \begin{cases}
 2^{ - n} \binom{{n}}{n/2}, &\text{if $n$ is even};\\ 
 0,&\text{if $n$ is odd}. \\ 
 \end{cases} 
\end{equation}
Identity~\eqref{knuth} occurs at $x=0$ in the following polynomial identity~\cite[Equation (6)]{adegoke26}:
\begin{equation}\label{l5xib79}
\sum_{k = 0}^n {( - 1)^{n - k} \binom{{n}}{k}2^{ - k} \binom{{2k}}{k}\left( {1 - x} \right)^{n - k} }  = \sum_{k = 0}^{\left\lfloor {n/2} \right\rfloor } {\binom{{n}}{{2k}}2^{ - 2k} \binom{{2k}}{k}x^{n - 2k} }.
\end{equation}
Starting with the following variation on~\eqref{l5xib79}:
\begin{equation}
\sum_{k = 0}^n {( - 1)^k \binom{{n}}{k}2^{ - k} \binom{{2k}}{k}x^k }  = \sum_{k = 0}^{\left\lfloor {n/2} \right\rfloor } {\binom{{n}}{{2k}}2^{ - 2k} \binom{{2k}}{k}x^{2k} \left( {1 - x} \right)^{n - 2k} }, 
\end{equation}
and proceeding as in the proof of Theorem~\ref{thm.touchard} yields the following generalization of Knuth's old sum.
\begin{theorem}\label{knuth-sum}
If $r$ and $s$ are complex numbers that are not negative integers, then
\begin{equation}\label{knuth-gen}
\sum_{k = 0}^n {( - 1)^k \binom{{n}}{k}2^{ - k} \binom{{2k}}{k}\binom{{k + r}}{s}^{-1}}  = \frac{s}{{n + r}}\sum_{k = 0}^{\left\lfloor {n/2} \right\rfloor } {\binom{{n}}{{2k}}2^{ - 2k} \binom{{2k}}{k}\binom{{n + r - 1}}{{2k + r - s}}^{-1}} .
\end{equation}
\end{theorem}
Equation~\eqref{knuth-gen} reduces to Knuth's old sum~\eqref{knuth} at $s=0$, while at $r=1=s$ we obtain the following binomial-transform identity:
\begin{equation}\label{it0ghdj}
\sum_{k = 0}^n {( - 1)^k \binom{{n}}{k}2^{ - k} C_k }  = 2^{ - n} \binom{{n}}{{\left\lfloor {n/2} \right\rfloor }},
\end{equation}
after using the fact that
\begin{equation}
\sum_{r = 0}^m {2^{ - 2r} \binom{{2r}}{r}}  = \sum_{r = 0}^m {\binom{{r - 1/2}}{r}}  = \binom{{m + 1/2}}{m} = \left( {2m + 1} \right)2^{ - 2m} \binom{{2m}}{m}.
\end{equation}
\begin{remark}
Identity~\eqref{it0ghdj} was derived on the MathWorld website~\cite{mathworld}.
\end{remark}
Choosing $r=0$ and $s=1/2$ in~\eqref{knuth-gen} and using
\begin{equation}
\binom k{1/2}=\frac{2^{2k+1}}\pi\,\binom{2k}k^{-1}
\end{equation}
and
\begin{equation}
\binom{{n - 1}}{{2k - 1/2}} = \frac{{2^{2n} }}{{n\pi }}\binom{{n}}{{2k}}\binom{{4k}}{{2k}}^{ - 1} \binom{{2\left( {n - 2k} \right)}}{{n - 2k}}^{ - 1} ,
\end{equation}
gives our next result.
\begin{proposition}
If $n$ is a non-negative integer, then
\begin{equation}
\sum_{k = 0}^n {( - 1)^k \binom{{n}}{k}2^{ - 3k} \binom{{2k}}{k}^2 }  = \frac{1}{{2^{2n} }}\sum_{k = 0}^{\left\lfloor {n/2} \right\rfloor } {2^{ - 2k} \binom{{2k}}{k}\binom{{4k}}{{2k}}\binom{{2\left( {n - 2k} \right)}}{{n - 2k}}} .
\end{equation}
\end{proposition}
The following variation on Equation~\eqref{l5xib79}:
\begin{equation}\label{twfyw7o}
\sum_{k = 0}^n {( - 1)^{n - k} \binom{{n}}{k}2^{ - k} \binom{{2k}}{k}x^{n - k} }  = \sum_{k = 0}^{\left\lfloor {n/2} \right\rfloor } {\binom{{n}}{{2k}}2^{ - 2k} \binom{{2k}}{k}\left(1-x\right)^{n - 2k} },
\end{equation}
and the application of Lemma~\ref{qqdca04} leads to another generalization of Knuth's old sum.

\begin{theorem}\label{knuth-gen-2}
If $n$ and $r$ are non-negative integers, then
\begin{align}
&\sum_{k = 0}^n {( - 1)^k \binom{{n}}{k}2^{ - k} \binom{{2k}}{k}\left( {n - k} \right)^r }\\  
&\qquad=\begin{cases}
 \sum\limits_{k = 0}^{\left\lfloor {r/2} \right\rfloor } {\binom{{n}}{{2k}}2^{ - (n - 2k)} \binom{{n - 2k}}{{\left( {n - 2k} \right)/2}}(2k)!\braces{{ r}}{{2k}}},&\text{if $n$ is even;}  \\ 
   \sum\limits_{k = 1}^{\left\lceil {r/2} \right\rceil } {\binom{{n}}{{2k-1}}2^{ - (n - 2k + 1)} \binom{{n - 2k + 1}}{{\left( {n - 2k + 1} \right)/2}}(2k - 1)!\braces{{ r}}{{2k - 1}}},&\text{if $n$ is odd.} 
 \end{cases}
\end{align}
\end{theorem}

Next, using~\eqref{twfyw7o} and Lemma~\ref{qqdca04}, we present a related result to Knuth's old sum involving a binomial-transform pair.
\begin{theorem}
Let $\left\{(a_k),(\alpha_k)\right\}$, $k=0,1,2,\ldots$, be a binomial-transform pair. Then
\begin{equation}
\sum_{k = 0}^n {( - 1)^{n - k} \binom{{n}}{k}2^{ - k} \binom{{2k}}{k}a_{n - k} }  = \sum_{k = 0}^{\left\lfloor {n/2} \right\rfloor } {\binom{{n}}{{2k}}2^{ - 2k} \binom{{2k}}{k}\alpha_{n - 2k} }.
\end{equation}

\end{theorem}
Finally, by considering the following variation on identity~\eqref{k24wp0d}:
\begin{equation}
\sum_{k = 0}^n {( - 1)^k \binom{{n}}{k}2^{ - k} C_{k + 1} x^k }  = \sum_{k = 0}^{\left\lfloor {n/2} \right\rfloor } {\binom{{n}}{{2k}}2^{ - 2k} C_k x^{2k}\left( {1 - x} \right)^{n - 2k} }, 
\end{equation}
and working as in the proof of Theorem~\ref{thm.touchard}, we derive the next combinatorial identity.
\begin{theorem}\label{thm.mzx3fur}
If $n$ is a non-negative integer and $r$ and $s$ are complex numbers that are not negative integers, then
\begin{equation}
\sum_{k = 0}^n {( - 1)^k \binom{{n}}{k}2^{ - k} C_{k + 1} \binom{{k + r}}{s}^{ - 1} }  = \frac{s}{{n + r}}\sum_{k = 0}^{\left\lfloor {n/2} \right\rfloor } {\binom{{n}}{{2k}}2^{ - 2k} C_k \binom{{n + r - 1}}{{2k + r - s}}} ^{ - 1} .
\end{equation}
In particular, setting $r=2$, $s=3/2$ and using the fact that
\begin{equation}
\binom{k+2}{3/2}=\frac{16}{3\pi}\frac{2^{2k}}{C_{k+1}}
\end{equation}
and
\begin{equation}
\binom{{n + 1}}{{2k + 1/2}} = \frac{{2^{2n + 2} }}{\pi }\frac{{n + 1}}{{\left( {2n - 4k + 1} \right)\left( {4k + 1} \right)}}\binom{{n}}{{2k}}\binom{{4k}}{{2k}}^{ - 1} \binom{{2\left( {n - 2k} \right)}}{{n - 2k}}^{ - 1}, 
\end{equation}
we obtain
\begin{align}
&\sum_{k = 0}^n {( - 1)^k \binom{{n}}{k}2^{ - 3k} C_{k + 1}^2 }\\
&\qquad  = \frac{1}{{\left( {n + 1} \right)\left( {n + 2} \right)2^{2n-1} }}\sum_{k = 0}^{\left\lfloor {n/2} \right\rfloor } {\left( {2n - 4k + 1} \right)\left( {4k + 1} \right)2^{ - 2k} C_k \binom{{4k}}{{2k}}\binom{{2\left( {n - 2k} \right)}}{{n - 2k}}} .
\end{align}
\end{theorem}

\end{document}